\documentclass[10pt]{amsart}
\usepackage{amssymb,amsbsy,amsmath,amsfonts,times, amscd}
\usepackage{latexsym,euscript,exscale}
\usepackage{graphicx,color} \usepackage{amsthm}
\usepackage{fancyhdr} \usepackage{url}
\usepackage{epigraph}\usepackage{lmodern}
\usepackage{mathrsfs}\usepackage{cancel}
\usepackage[toc,page]{appendix}\usepackage[T1]{fontenc}
\usepackage{mathtools}\usepackage[all,cmtip]{xy}
\usepackage{enumerate}\usepackage{setspace}
\usepackage{helvet}
\usepackage{float}

\numberwithin{equation}{section}
\newtheorem{theorem}{Theorem}
\newtheorem{thm}[theorem]{Theorem}

\newtheorem{cor}[theorem]{Corollary}
\newtheorem{lemma}[theorem]{Lemma}
\newtheorem{prop}[theorem]{Proposition}
\theoremstyle{definition}
\newtheorem{defi}[theorem]{Definition}
\newtheorem{problem}[theorem]{Problem}
\theoremstyle{remark}
\newtheorem{remark}[theorem]{Remark}

\numberwithin{theorem}{section}

\newcommand {\M}{\mathbb M}

\newcommand{\eps}{\varepsilon}

\newcommand{\N}{\mathbb{N}}
\newcommand{\R}{\mathbb{R}}
\newcommand{\Z}{\mathbb{Z}}

\newcommand{\tn}{{\vert\kern-0.25ex\vert\kern-0.25ex\vert}}
\newcommand{\Tn}{{\big\vert\kern-0.25ex\big\vert\kern-0.25ex\big\vert}}    
\newcommand{\TN}{{\Big\vert\kern-0.25ex\Big\vert\kern-0.25ex\Big\vert}} 
\begin{document}

\title[Nonlinear weakly sequentially continuous  embeddings.]{Nonlinear weakly sequentially continuous embeddings between Banach spaces.}
 
\author{B. M. Braga}
\address{Department of Mathematics and Statistics\\
 York University\\
4700 Keele St.\\ Toronto, Ontario, M3J IP3\\
Canada}\email{demendoncabraga@gmail.com}

\date{}
\maketitle
\renewcommand{\thefootnote}{*}
\footnotetext{  This research was partially supported by Simons Foundation via the York Science Research Fellowship.}

\begin{abstract}
In these notes, we study nonlinear embeddings between Banach spaces which are also weakly sequentially continuous. In particular, our main result implies  that if a Banach space $X$  coarsely (resp. uniformly) embeds into a Banach space $Y$ by a weakly sequentially continuous map,  then every spreading model $(e_n)_n$ of a normalized weakly null sequence in $X$ satisfies
\[\|e_1+\ldots+e_k\|_{\overline{\delta}_Y}\lesssim\|e_1+\ldots+e_k\|_S,\]
where $\overline{\delta}_Y$ is the modulus of asymptotic uniform convexity of $Y$. Among many other results, we obtain  Banach spaces $X$ and $Y$ so that $X$  coarsely (resp. uniformly) embeds into $Y$,  but so that $X$ cannot be mapped into $Y$ by a weakly sequentially continuous  coarse (resp. uniform) embedding.
\end{abstract}

\section{Introduction.}\label{SectionIntro}

The study of uniform and coarse embeddability is an active area of research in the nonlinear geometry of Banach spaces (see for example  \cite{BaudierLancienSchlumprecht2017}, \cite{BragaSwift}, \cite{Kalton2007}, \cite{Kalton2013Israel}, and \cite{MendelNaor2008}). Let $X$ and $Y$ be Banach spaces and $f:X\to Y$ be a map. We define the \emph{modulus of uniform continuity}  and the \emph{modulus of expansion} of $f$ by letting
\[\omega_f(t)=\sup\{\|f(x)-f(y)\|\mid \|x-y\|\leq t\},\]
and
\[\rho_f(t)=\inf\{\|f(x)-f(y)\|\mid \|x-y\|\geq t\},\]
for all $t\geq 0$, respectively. We say that $f$ is a \emph{coarse map} if $\omega_f(t)<\infty$, for all $t\geq 0$, and  we say that $f$ is a \emph{coarse embedding} if it is  coarse and  $\lim_{t\to\infty}\rho_f(t)=\infty$. The map $f$ is a uniform embedding if $\lim_{t\to 0^+}\omega_f(t)=0$ and $\rho_f(t)>0$, for all $t>0$.

In these notes, we introduce a new type of nonlinear embedding between Banach spaces and give several results regarding the preservation of the asymptotic geometry of Banach spaces under this embedding. More precisely, this paper is focused in the study of  weakly sequentially continuous  maps $f:X\to Y$ which are coarse and satisfy the following property
\begin{equation}\label{TheProperty}
\exists \beta>\alpha>0\ \ \text{such that}\ \ \inf_{\|x-y\|\in[\alpha,\beta]}\|f(x)-f(y)\|>0.\tag{$\ast$}
\end{equation}
Since uniform embeddings between Banach spaces are automatically coarse maps (see \cite{Kalton2008}, Lemma 1.4), it follows that the property of a map $f:X\to Y$ being both coarse and satisfying Property $(*)$ is weaker than $f$ being either a  coarse or a uniform embedding. We point out that other  weakenings of coarse and uniform embeddability have been studied in \cite{Braga3} and \cite{Rosendal2017}.

It is known that if $f:X\to Y$ is any map, then for every $\eps>0$ there exists a continuous map $F:X\to Y$ so that
\[\sup_{x\in X}\|f(x)-F(x)\|\leq \inf_{t>0}\omega_f(t)+\eps\]
(see \cite{Braga1}, Theorem 1.4). In particular, if $f$ is a coarse embedding, so is $F$. Hence, a Banach space $X$ coarsely embeds into another Banach space $Y$ if and only if it coarsely embeds by a map which is also continuous. Therefore, it is natural to wonder which  topologies we can endow $X$ and $Y$ with in order to obtain a similar result. Precisely, we have the following general question.

\begin{problem}\label{probTop}
Let $X$ and $Y$ be Banach spaces and $\tau_X$ and $\tau_Y$ be topologies in $X$ and $Y$, respectively. Assume that $X$ coarsely (resp. uniformly) embeds into $Y$. Can we find a coarse (resp. uniform) embedding $(X,\tau_X)\to (Y,\tau_Y)$ which is also continuous?
\end{problem}

Among other applications, the main result in this paper (see Theorem \ref{MainResult} bellow) allows us to give a negative answer to Problem \ref{probTop} if  $\tau_X$ and $\tau_Y$ are the weak topologies of $X$ and $Y$, respectively. Not surprisingly,  the Banach spaces' asymptotic structures will play an important role in the study of weakly sequentially continuous maps $X\to Y$. 

Given a Banach space $X$,  the \emph{modulus of asymptotic uniform convexity of $X$} is given by
\[\overline{\delta}_X(t)=\inf_{x\in\partial B_X}\sup_{E\in \text{cof}(X)}\inf_{y\in\partial B_E}\|x+ty\|-1.\]
In \cite{Kalton2013},  N. Kalton studied the relation between  the modulus of asymptotically uniform convexity and coarse Lipschitz embeddings. Let $X$ and $Y$ be Banach spaces. We say that $X$ \emph{coarse Lipschitz embeds} into $Y$ if there exist $f:X\to Y$ and $L>0$ so that
\[\omega_f(t)\leq Lt+L,\ \ \text{and}\ \ \rho_f(t)\geq L^{-1}t-L, \ \ \text{for all}\ \ t\geq 0.\]  In \cite{Kalton2013}, Theorem 7.4, Kalton proved Theorem \ref{KaltonKalton} below which was a breakthrough in the study of coarse Lipschitz embeddings between Banach spaces. In what follows, $\|e_1+\ldots+e_k\|_{\overline{\delta}_Y}\coloneqq \inf\{\lambda>0\mid k\cdot\overline{\delta}_Y(\lambda^{-1})\leq 1\}$ (see Subsection \ref{SubsectionOrlicz} for a formal definition of $\|\cdot\|_{\overline{\delta}_Y}$).

\begin{thm}[\textbf{N. Kalton, 2013}]\label{KaltonKalton}
Let $X$ and $Y$ be Banach spaces and assume that $X$ coarse Lipschitz embeds into $Y$. Then, there exists $c>0$ so that if $(e_n)_n$ is a spreading model in a Banach space $(S,\|\cdot\|_S)$ of a normalized weakly null sequence in $X$ then 
\[c\|e_1+\ldots+e_k\|_{\overline{\delta}_Y}\leq \|e_1+\ldots+e_k\|_S,\]
for all $k\in \N$.
\end{thm}

We refer the reader to Subsection \ref{SubsectionSpreading} for  definitions regarding  spreading models in Banach spaces.

Since $\ell_q$ coarsely embeds into $\ell_p$, for all $p,q\in [1,2]$ (see \cite{Nowak2006}, Theorem 5), we cannot weaken the hypothesis that $X$ coarse Lipschitz embeds into $Y$ in Kalton's theorem and simply assume coarse embeddability of $X$ into $Y$. However, the situation is different if we look at weakly sequentially continuous maps.  Using ideas present in \cite{BaudierLancienSchlumprecht2017}, we can show the following theorem, which is the main result of these notes.

\begin{thm}\label{MainResult}
Let $X$ and $Y$ be Banach spaces and assume that $X$ maps into $Y$ by a weakly sequentially continuous map which is coarse and satisfies Property $(*)$.  There exists $c>0$, so that if $(e_n)_n$ is the spreading model in a Banach space $(S,\|\cdot\|_S)$ of a normalized weakly null sequence in $X$, then
\[c\|e_1+\ldots+e_k  \|_{\overline{\delta}_Y}\leq\|e_1+\ldots+e_k\|_S,\]
for all $k\in\N$. 
\end{thm}

A Banach space $X$ is \emph{asymptotically uniformly convex} (\emph{AUC} for short) if $\overline{\delta}_X(t)>0$, for all $t\in [0,1]$. For $p\in[1,\infty)$, we say that $X$ is  \emph{$p$-asymptotically uniformly convex} (\emph{$p$-AUC} for short) if there exists $K>0$ so that $\overline{\delta}_X(t)>Kt^p$, for all $t\in [0,1]$ (see Section \ref{SectionBackground} for more on asymptotic uniform convexity). If $X$ has an equivalent norm in which $X$ is AUC (resp. $p$-AUC), we say that the Banach space $X$ is \emph{AUCable} (resp. \emph{$p$-AUCable}). We say that a Banach space $X$ has the \emph{$p$-co-Banach-Saks property} if there exists $C>0$ so that every spreading model $(e_n)_n$ in a Banach space $(S,\|\cdot\|_S)$ of a normalized weakly null sequence in $X$ satisfies
\[\|e_1+\ldots+e_k\|_S\geq Ck^{1/p},\]
for all $k\in\N$. Theorem \ref{MainResult} gives us the following corollary.

\begin{cor}\label{CorMainResult}
Let $X$ and $Y$ be Banach spaces and assume that $X$ maps into $Y$ by a weakly sequentially continuous map which is coarse and satisfies Property $(*)$. Let $p\in [1,\infty)$. If $Y$ is $p$-AUC, then $X$ must have the $p$-co-Banach-Saks property.
\end{cor}

We point out that Corollary \ref{CorMainResult} remains true if  $Y$  satisfies asymptotic lower $\ell_p$-estimates. We refer the reader to Subsection \ref{SubsectionLowerEst}  for the precise definition of asymptotic lower $\ell_p$-estimates and we refer to  Theorem \ref{ThmAsymp} for a precise statement of this result.

It follows straightforwardly  from Corollary \ref{CorMainResult} that if a Banach space $X$ without the $p$-co-Banach-Saks property can be mapped into a Banach space $Y$ by a weakly sequentially continuous  coarse map satisfying Property $(*)$, then $Y$ is not $p$-AUCable. In particular, if $X$  does not have the $p$-Banach-Saks property for all $p\in [1,\infty)$, then $Y$ is not $p$-AUCable for all $p\in[1,\infty)$. To the best of our knowledge, it is not known whether an AUC space is $p$-AUCable, for some $p\in[1,\infty)$. However, we are able to obtain the following.

\begin{thm}\label{c0spread2}
Let $X$ and $Y$ be  Banach spaces and assume that  $X$ has a normalized weakly null sequence whose spreading model is isomorphic to the standard basis of $c_0$. Assume that there exists a weakly sequentially continuous map $f:X\to Y$ which is coarse and satisfies Property $(*)$. Then $Y$ is not AUCable.
\end{thm}

Using results of \cite{Kalton2013}, Corollary \ref{CorMainResult} allow us to obtain the following.

\begin{cor}\label{CorL_p}
Let $p\in (1,2)$, and assume that $X$ maps into $Y$ by a weakly sequentially continuous map which is coarse and satisfies Property $(*)$. We have the following.
\begin{enumerate}[(i)]
\item If $X$ is a subspace of $L_p$, then  $X$ is isomorphic to a subspace of $\ell_p$. 
\item  If $X$ is an $\mathcal{L}_p$-space, then $X$ is isomorphic to  $\ell_p$.
\end{enumerate} 
\end{cor}

Since this paper introduces a, at least, formally new kind of embedding, it is important to notice that this is indeed a new kind of embedding. As it is well known that $\ell_q$ both coarsely and uniformly  embeds into $\ell_p$, for all $p,q\in [1,2]$ (see \cite{Nowak2006}, Theorem 5), we obtain the following.

\begin{cor}\label{ellpdoesnotellq}
Suppose $1 \leq p<q$. Then there exists no weakly sequentially continuous map $\ell_q\to\ell_p$ which is coarse and satisfies Property $(*)$. In particular, the existence of a weakly sequentially continuous coarse (resp. uniform) embedding  between Banach spaces is strictly stronger than  coarse (resp. uniform) embeddability.
\end{cor}

Still in the topic of whether weakly sequentially continuous coarse maps satisfying Property $(*)$ is  actually a new kind of embedding, one could wonder whether  demanding that a coarse (resp. uniform) embedding is weakly sequentially continuous is too restrictive. Indeed, it could be the case that the existence of such   embedding implies the existence of an  isomorphic embedding. Fortunately, this is not the case. More precisely, since we are looking at weakly sequentially continuous maps $X\to Y$, it would be useful to find sufficient conditions for a map $X\to Y$ to be weakly sequentially continuous. With that in mind, we prove Lemma \ref{lemmaImpliesWeakGEN} in Section \ref{SectionExample}, which gives us Theorem \ref{222} below. A map $f:X\to Y$ is a \emph{strong embedding} if $f$ is both a coarse and a uniform embedding.

 \begin{thm}\label{222}
Let $1\leq p<q$. There exists a strong embedding  $\ell_p\to \ell_q$ which is also weakly sequentially continuous. In particular, the existence of a  weakly sequentially continuous coarse (resp. uniform) embedding  between Banach spaces is strictly weaker than the existence of an isomorphic embedding.
 \end{thm}

This paper is organized as follows. In Section \ref{SectionBackground}, we introduce the remaining notations and definitions required for this paper. In Section \ref{SectionMainResults}, we use ideas in \cite{BaudierLancienSchlumprecht2017} and \cite{KaltonRandrianarivony2008} in order to prove a general lemma (see Lemma \ref{lemmaGen} below) which we use to prove Theorem \ref{MainResult}. In Subsection \ref{SubsectionLowerEst}, we define asymptotic games in order to obtain a version of Corollary \ref{CorMainResult} to spaces satisfying asymptotic lower $\ell_p$-estimates (see Theorem \ref{ThmAsymp} below). In the case of Banach spaces $Y$ with separable dual, Theorem \ref{ThmAsymp} generalizes Corollary \ref{CorMainResult} above. Subsection \ref{SubsectionWeakStar} is dedicated to noticing that, in the context of dual Banach spaces, our results have weak$^*$ analogs. At last, in Section \ref{SectionExample}, we give a sufficient condition for a coarse map $X\to Y$ to be weakly sequentially continuous (under some assumptions on the target space $Y$) and use this to show that $\ell_p$ strongly embeds into $\ell_q$ by a weakly sequentially continuous map, for $1\leq p<q<\infty$.

At last, we would like to point out that, while this paper deals with a notion of nonlinear embedding between Banach spaces and how it relates to asymptotic uniform convexity, some work has been done on nonlinear embeddings between Banach spaces and asymptotic uniform smoothness (see \cite{BaudierLancienSchlumprecht2017},  \cite{Braga2},  and \cite{KaltonRandrianarivony2008}).

\section{Background.}\label{SectionBackground}
Let $X$ be a Banach space. We denote its closed unit ball by $B_X$ and its closed unit sphere by $\partial B_X$. In these notes, all subspaces of a given Banach space are assumed to be norm closed. We denote the set consisting of all finite codimensional subspaces of $X$ by $\text{cof}(X)$. Given a subset $A\subset \N$ and $k\in\N$, we denote by $[A]^k$ the set of all subsets of $A$ with $k$ elements. Given $\bar{n}\in [A]^k$, we always write $\bar{n}=(n_1,\ldots,n_k)$ in increasing order, i.e., with $n_1<\ldots<n_k$. 

If $X$ and $Y$ are Banach spaces,  we write $X\equiv Y$ if $X$ and $Y$ are linearly isometric to each other. We say that a map $f:X\to Y$ is \emph{weakly sequentially continuous} if for every sequence $(x_n)_n$ in $X$ which weakly converges to $x\in X$ it follows that $(f(x_n))_n$ weakly converges to $f(x)$.

In these notes, we make repetitive use of the following version of Elton's Near Unconditionality Theorem (see \cite{Elton1978}).

\begin{thm}[\textbf{J. Elton, 1978}]\label{Elton}
There exists $D>0$ with the following property: every normalized weakly null sequence  $(x_n)_n$ in a Banach space $X$ has a subsequence $(y_n)_n$ so that
\[\Big\|\sum_{j=1}^k\eps_jy_{n_j}\Big\|\leq D \Big\|\sum_{j=1}^k \delta_jy_{n_j}\Big\|,\] 
for all $n_1<\ldots<n_k\in \N$ and all  $\eps_1,\ldots,\eps_k,\delta_1,\ldots,\delta_k\in\{-1,1\}$. 
\end{thm}

Let $X$ and $Y$ be Banach spaces and $f:X\to Y$ be a coarse map. Since $X$ is metrically convex, it follows that there exists $L>0$ so that $\omega_f(t)\leq Lt+L$ (see \cite{Kalton2008}, Lemma 1.4). Therefore, for all $\eps>0$, there exists $L>0$ so that 
\[\|f(x)-f(y)\|\leq L\|x-y\|,\ \ \text{if}\ \ \|x-y\|\geq \eps.\]

Let $X$ be a Banach space and let $\overline{\delta}_X$ be the modulus of asymptotic convexity of $X$ (see Section \ref{SectionIntro} for definitions). It is easy to see that $\overline{\delta}_X$ is $1$-Lipschitz, and that for every weakly null sequence $(x_{n})_{n\in\N}$ in $X$, every nonprincipal ultrafilter $\mathcal{U}$ on $\N$ and every $x\in X\setminus \{0\}$, we have that
\[\|x\|\lim_{n,\mathcal{U}}\overline{\delta}_X\Big(\frac{\|x_\lambda\|}{\|x\|}\Big)\leq \lim_{n,\mathcal{U}}\|x+x_\lambda\|-\|x\|.\]

\subsection{Orlicz sequence spaces.}\label{SubsectionOrlicz}
Let $F:[0,\infty)\to[0,\infty)$ be a continuous map so that $F(0)=0$. We define a set $\ell_F\subset \R^\N$ by setting 
 \[\ell_F=\Big\{(x_n)_n\in\R^\N\mid \exists\lambda>0,\ \sum_{j=1}^\infty F\Big(\frac{|x_n|}{\lambda}\Big)<\infty\Big\}.\]
For each $(x_n)_n\in \ell_F$, let 
\[\|(x_n)_n\|_{F}\coloneqq\inf\Big\{\lambda>0\mid  \sum_{j=1}^\infty F\Big(\frac{|x_n|}{\lambda}\Big)\leq 1\Big\}.\]
Notice that, besides the notation being used here, $\|\cdot\|_F$ does not need to be a norm on $\ell_F$.  We denote by $(e_n)_n$ the sequence in $\ell_F$ so that the $n$'th coordinate of $e_n\in \ell_F$ is $1$ and its $j$'th coordinate is $0$,  for all $n\neq j$. 

A nonzero function $F:[0,\infty)\to[0,\infty)$ is an \emph{Orlicz function} if it is continuous, nondecreasing,  convex, and  satisfies $F(0)=0$. It is well know that if $F$ is an Orlicz function, then $\|\cdot\|_F$ is a norm on $\ell_F$ which makes $\ell_F$ into a Banach space. The space $(\ell_F,\|\cdot\|_F)$ is called the \emph{Orlicz sequence space associated to $F$}. We refer to \cite{LindenstraussTzafriri1971} for more on Orlicz sequence spaces.

Let $X$ be a Banach space. Although $\overline{\delta}_X$ is not a convex function in general, the function $\overline{\delta}_X(t)/t$ is increasing. Therefore, we can define a convex function by letting
\[\tilde{\delta}_X(t)=\int_{0}^t\frac{\overline{\delta}_X(s)}{s}ds,\]
for all $t\geq 0$. So, $\tilde{\delta}_X$ is a $1$-Lipschitz Orlicz function and we can define the Orlicz sequence space $\ell_{\tilde{\delta}_X}$. Notice that $\overline{\delta}_X(t/2)\leq \tilde{\delta}_X(t)\leq \overline{\delta}_X(t)$, for all $t\geq 0$, i.e., the functions $\overline{\delta}_X$ and $\tilde{\delta}_X$ are \emph{equivalent}. In particular, we have that
\begin{equation}\label{Ohoh}
\frac{1}{2}\|e_1+\ldots +e_k\|_{\overline{\delta}_Y}\leq \|e_1+\ldots +e_k\|_{\tilde{\delta}_Y}\leq \|e_1+\ldots +e_k\|_{\overline{\delta}_Y},
\end{equation}
 for all $k\in\N$. The space $(\ell_{\tilde{\delta}_X},\|\cdot\|_{\tilde{\delta}_X})$ will play an important role in this paper.

\subsection{Spreading models.}\label{SubsectionSpreading}
 
Let $(X,\|\cdot\|)$ be a Banach space and $(x_n)_n$ be a bounded sequence without Cauchy subsequences. There exists a  sequence $(e_n)_n$ in a Banach space $(S,\|\cdot\|_S)$ so that  for all $\eps>0$ and all $k\in\N$, there exists $\ell\in\N$ such that 
\[\Bigg| \Big\|\sum_{i=1}^ka_i x_{n_i}\Big\|-\Big\|\sum_{i=1}^ka_i e_{i}\Big\|_S\Bigg|\leq \eps,\]
for all $\ell\leq n_1<\ldots<n_k$ and all $a_1,\ldots,a_k\in [-1,1]$  (see \cite{G-D}, Chapter 2, Section 2, for a proof of this fact).  
The sequence $(e_n)_n$ is called the \emph{spreading model of the sequence $(x_n)_n$}.

\begin{remark}
Let $X$ be a Banach space and $F:[0,\infty)\to[0,\infty)$ be a continuous function so that $F(0)=0$. We would like to point out that we  abuse of notation, and we  use $(e_n)_n$ to denote both the spreading model of a sequence in a Banach space $X$ and the unit sequence in  $\ell_F$ defined in Subsection \ref{SubsectionOrlicz}. 
\end{remark}

\section{Restrictions for the existence of  weakly sequentially continuous embeddings.}\label{SectionMainResults}

In order to prove Theorem \ref{MainResult}, we need to prove some  technical  lemmas first. Fix a Lipschitz Orlicz function $F:[0,\infty)\to[0,\infty)$. By Fekete's lemma, the limit $\theta= \lim_{t\to\infty}F(t)/t$ exists and it is easy to see that $\theta>0$. We define a sequence of  norms $(N_k)_k=(N^F_k)_k$ as follows. Let $N_2= N^F_2$ be a map on $\R^2$ given by 
\[N_2(\xi_1,\xi_2)=\left\{\begin{array}{l l}
|\xi_1|F\Big(\frac{|\xi_2|}{|\xi_1|}\Big)+|\xi_1|, &\xi_1\neq 0\\
\theta|\xi_2|, &\xi_1=0.
\end{array}\right.\]
By the definition of $\theta$, we have that $N_2$ is continuous, and as $F$ is convex it follows that $N_2$ is a norm on $\R^2$. We now proceed by induction. Suppose $k\geq 3$ and that $N_{k-1}= N^F_{k-1}:\R^{k-1}\to \R$ has already been defined. We define $N_k= N^F_k:\R^k\to \R$ by letting 
\[N_k(\xi_1,\ldots,\xi_k)=N_2(N_{k-1}(\xi_1,\ldots,\xi_{k-1}),\xi_k),\ \ \text{for all}\ \ \xi_1,\ldots,\xi_k\in\R.\]
We also denote the usual norm on $\R$ by $N_1$, i.e., $N_1(\xi)=|\xi|$, for all $\xi\in\R$. So, $N_2(\xi_1,\xi_2)=N_2(N_1(\xi_1),\xi_2)$ also holds, for all $\xi_1,\xi_2\in \R$.

Norms of this type were first considered in \cite{Kalton1993}, and they were used to solve problems regarding asymptotic uniform smothness  in both \cite{KaltonRandrianarivony2008} and \cite{Kalton2013}. It is clear that the $N_k$'s are \emph{absolute norms}, i.e., 
\[ N_k(\xi_1,\ldots,\xi_k)=N_k(|\xi_1|,\ldots,|\xi_k|),\]
for all $\xi_1,\ldots,\xi_k\in \R$. Hence, it follows that the unit vectors of $\R^k$ form a $1$-unconditional basis for $\R^k$ endowed with the norm $N_k$, i.e., 
\[ N_k(\xi_1,\ldots,\xi_k)\leq N_k(\zeta_1,\ldots,\zeta_k),\]
for all $\xi_1,\zeta_1,\ldots,\xi_k,\zeta_k\in\R$, with $|\xi_j|\leq |\zeta_j|$, for all $j\in\{1,\ldots,k\}$ (see \cite{Kalton1993}, Proposition 3.2(1)).

\begin{lemma}\label{Nkdes}
Let $F:[0,\infty)\to[0,\infty)$ be a Lipschitz Orlicz function  so that $F(1)\leq 1$ and let $(N_k)_k\coloneqq (N^F_k)_k$ be defined as above. For all $k\in\N$, we have that
\[N_k(\underbrace{1,\ldots,1}_{k\text{-times}})\geq \|e_1+\ldots+e_k\|_{\tilde{\delta}_Y}.\]
\end{lemma}

We point out that Kalton proved in \cite{Kalton2013}, Lemma 4.3, that $\|e_1+\ldots+e_k\|_F\leq 2N_k(1,\ldots,1)$, for all $k\in\N$. Although Kalton's result would be enough for our goals, since we obtain a better constant and for the convenience of the reader, we present the proof of Lemma \ref{Nkdes}.

\begin{proof}[Proof of Lemma \ref{Nkdes}]
We proceed by induction on $k\in\N$. For $k=1$ the result follows since $F(1)\leq 1$ implies $\|e_1\|_F\leq 1=N_1(1)$. Assume the result holds for $k-1$.  Then we have 
\begin{align*}
N_k(1,\ldots,1)&=N_2(N_{k-1}(1,\ldots,1),1)\\
&\geq N_2(\|e_1+\ldots+e_{k-1}\|_F,1)\\
&= \|e_1+\ldots+e_{k-1}\|_F
\Big(1+F\Big(\frac{1}{\|e_1+\ldots+e_{k-1}\|_F}\Big)\Big).
\end{align*}
Since $F$ is continuous, by the definition of the norm $\|\cdot\|_F$, we must have 
\[(k-1)F\Big(\frac{1}{\|e_1+\ldots+e_{k-1}\|_F}\Big)=1.\]
So, 
\[N_k(1,\ldots,1)\geq \|e_1+\ldots+e_{k-1}\|_{F}
\Big(1+\frac{1}{k-1}\Big).\]
We need a general lemma.

\begin{lemma}\label{111}
Let $(e_n)_n$ be a sequence in a Banach space satisfying $\|e_1+\ldots+e_k\|=\|e_{n_1}+\ldots+e_{n_k}\|$, for all $n_1<\ldots<n_k\in\N$. For all  $k\in\N$, the following holds
\[\|e_1+\ldots+e_k\|\leq  \|e_1+\ldots+e_{k-1}\|
\Big(1+\frac{1}{k-1}\Big).\]
\end{lemma}

\begin{proof}
Suppose the inequality above does not hold for some $k\in\N$. So, 
\[\Big(1+\frac{1}{k-1}\Big)\|e_{n_1}+\ldots+e_{n_{k-1}}\|<
\|e_1+\ldots+e_k\|,\]
for all $\bar{n}=(n_1,\ldots,n_{k-1})\in[\{1,\ldots,k\}]^{k-1}$. Then

\begin{align*}
k\|e_1+\ldots+e_k\|& =\Big(1+\frac{1}{k-1}\Big)\|(k-1)e_1+\ldots+(k-1)e_k\|\\
&\leq \Big(1+\frac{1}{k-1}\Big)\sum_{\bar{n}\in[\{1,\ldots,k\}]^{k-1}}\|e_{n_1}+\ldots+e_{e_{k-1}}\|\\
&< k\|e_1+\ldots+e_k\|, 
\end{align*}
which gives us a contradiction.
\end{proof}
Since $(e_n)_n$  satisfies the hypothesis of Lemma \ref{111}, the result now clearly follows for $k$, and our induction is complete.
\end{proof}

Let $X$ be a Banach space and $\beta>\alpha>0$. We say that a sequence $(x_n)_n$ in $X$ is \emph{$[\alpha,\beta]$-separated} if $\|x_n-x_m\|\in[\alpha,\beta]$, for all $n\neq m$.

\begin{lemma}\label{lemmaGenNOVO}
Let $X$ and $Y$ be  Banach spaces, and let $(N_k)_k\coloneqq (N^{\tilde{\delta}_Y}_k)_k$. Let $f:X\to Y$ be weakly sequentially continuous  and let  $\beta>\alpha>0$. Given an $[\alpha,\beta]$-separated weakly convergent  sequence  $(x_n)_n$ and $k\in\N$, we define    $h_k:[\N]^k\to Y$ by letting \[h_k(\bar{m})=f(x_{m_1}+\ldots+x_{m_k}),\] for all  $\bar{m}=(m_1,\ldots,m_k)\in[\N]^k$. Then for every $\eps>0$ and every infinite subset $\M\subset\N$ there exist $\bar{m}=(m_1,\ldots,m_k),\bar{n}=(n_1,\ldots,n_k)\in[\M]^k$, with $m_1<n_1<\ldots<m_k<n_k$, so that
\[\|h_k(\bar{m})-h_k(\bar{n})\|\geq N_k\Big(\frac{b}{2},\ldots,\frac{b}{2}\Big)-\eps,\]
where 
\[b\coloneqq\inf_{\|x-y\|\in[\alpha,\beta]}\|f(x)-f(y)\|.\]
\end{lemma}

Lemma \ref{lemmaGenNOVO} will be a simple consequence of the following more elaborated lemma. We point out that the proof of Lemma \ref{lemmaGen} below is inspired in Proposition 4.1 of \cite{BaudierLancienSchlumprecht2017}.

\begin{lemma}\label{lemmaGen}
Let $X$ and $Y$ be  Banach spaces, and let $(N_k)_k\coloneqq (N^{\tilde{\delta}_Y}_k)_k$. Let $f:X\to Y$ be a   weakly sequentially continuous. Let  $\beta>\alpha>0$. Given an $[\alpha,\beta]$-separated weakly convergent  sequence  $(x_n)_n$ and $k\in\N$, we define    $h_k:[\N]^k\to Y$ by letting \[h_k(\bar{m})=f(x_{m_1}+\ldots+x_{m_k}),\] for all  $\bar{m}=(m_1,\ldots,m_k)\in[\N]^k$. 

The following holds: for all $k\in\N$,  all $\eps>0$, and all $[\alpha,\beta]$-separated weakly convergent  sequence  $(x_n)_n$ in $X$,  there exists  $y\in Y$ so that for all infinite subset $\M\subset \N$, 
\[\forall Y_1\in\text{cof}(Y),\ \exists y_1,z_1\in Y_1,\ \ldots, \ \forall Y_k\in\text{cof}(Y),\ \exists y_k,z_k\in Y_k,\] there exist $\bar{m}=(m_1,\ldots,m_k),\bar{n}=(n_1,\ldots,n_k)\in[\M]^k$, with $m_1<n_1<\ldots<m_k<n_k$, so that

\begin{enumerate}[(i)]
\item $\|h_k(\bar{m})-(y+y_1+\ldots+y_k)\|\leq \eps$,
\item $\|h_k(\bar{n})-(y+z_1+\ldots+z_k)\|\leq \eps$, 
\item $\|y_i-z_i\|\geq b/2-\eps$, for all $i\in\{1,\ldots,k\}$, and
\item $\|y_1-z_1+\ldots+y_k-z_k\|\geq N_k\Big(\frac{b}{2},\ldots,\frac{b}{2}\Big)-\eps$,
\end{enumerate}
where 
\[b\coloneqq\inf_{\|x-y\|\in[\alpha,\beta]}\|f(x)-f(y)\|.\]
\end{lemma}

\begin{proof}
Assume $b>0$, and let us prove the result. The case $b=0$ follows similarly. Let us proceed by induction on $k$.    Suppose $k=1$, and let $(x_n)_n$ be an $[\alpha,\beta]$-separated  sequence weakly converging to $x\in X$. As $f$ is  weakly sequentially continuous,  $(h_1(n))_n$ converges weakly to some $y=f(x)\in Y$. As $\|x_n-x_m\|\in[\alpha, \beta]$, for all $m\neq n$, we have that $\|h_1(n)-h_1(m)\|\geq b$, for all $n\neq m$. So, we can assume that $\|h_1(n)-y\|\geq b/2$, for all $n\in\N$. 

Let $\M\subset \N$ be an infinite subset. As $(h_1(n)-y)_{n\in\M}$ is weakly null, for each $Y_1\in\text{cof}(X)$, we can pick $y_1,z_1\in Y_1$ and $m_1<n_1\in\M$ so that $\|h_1(m_1)-y-y_1\|<\eps$, $\|h_1(n_1)-y-z_1\|<\eps$ and $\|y_1\|\leq \|y_1-z_1\|$. In particular, $\|y_1\|,\|z_1\|> b/2-\eps$. Hence, $\|y_1-z_1\|> b/2-\eps$, and the statement follows for $k=1$.

Assume  that our statement is proved for $k-1$, let us prove it for $k$. Let $(x_n)_n$ be an $[\alpha,\beta]$-separated  weakly convergent sequence, say $x=w\text{-}\lim_nx_n$. For each $\bar{m}\in[\N]^{k-1}$, let $\tilde{h}(\bar{m})=f(x_{m_1}+\ldots+ x_{m_{k-1}}+x)$. As $f$ is  weakly sequentially continuous, we have that $\tilde{h}(\bar{m})=w\text{-}\lim_n h_k(\bar{m},n)$, for all $\bar{m}\in[\N]^{k-1}$. As $(x_n+x/(k-1))_n$ is an $[\alpha,\beta]$-separated weakly convergent sequence, we can apply the induction hypothesis to obtain  $y\in Y$ so that for every infinite subset $\M\subset \N$, 
\[\forall Y_1\in\text{cof}(Y),\ \exists y_1,z_1\in Y_1,\ \ldots, \ \forall Y_{k-1}\in\text{cof}(Y),\ \exists y_{k-1},z_{k-1}\in Y_{k-1},\] 
there exist $\bar{m}_0=(m_1,\ldots,m_{k-1}),\bar{n_0}=(n_1,\ldots,n_{k-1})\in[\M]^{k-1}$, with $m_1<n_1<\ldots<m_{k-1}<n_{k-1}$, so that

\begin{enumerate}[(i)]
\item $\|\tilde{h}(\bar{m}_0)-(y+y_1+\ldots+y_{k-1})\|\leq \delta$, 
\item $\|\tilde{h}(\bar{n}_0)-(y+z_1+\ldots+z_{k-1})\|\leq \delta$, \item $\|y_i-z_i\|\geq b/2-\delta$, for all $i\in\{1,\ldots,k-1\}$, and
\item $\|y_1-z_1+\ldots+y_{k-1}-z_{k-1}\|\geq N_{k-1}\Big(\frac{b}{2},\ldots, \frac{b}{2}\Big)-\delta$,
\end{enumerate}
where $\delta>0$ is chosen so that $\delta<\eps/2$ and 
\[N_2\Big(N_{k-1}\Big(\frac{b}{2},\ldots, \frac{b}{2}\Big)-\delta,\frac{b}{2}-\delta\Big)-3\delta\geq N_2\Big(N_{k-1}\Big(\frac{b}{2},\ldots, \frac{b}{2}\Big),\frac{b}{2}\Big)-{\eps}.\]

Fix an infinite subset $\M\subset \N$ and let \[Y_1,\ldots,Y_{k-1}\in\text{cof}(Y),\ \  y_1,z_1,\ldots,y_{k-1},z_{k-1}\in Y,\] and $\bar{m}_0,\bar{n}_0\in[\M]^{k-1}$ be given as above. Fix $Y_k\in \text{cof}(Y)$. Since $\|h_k(\bar{m}_0,n)-h_k(\bar{m}_0,m)\|\geq b$ and $\|h_k(\bar{n}_0,n)-h_k(\bar{n}_0,m)\|\geq b$, for all $n\neq m$, we can assume that $\|h_k(\bar{m}_0,n)-\tilde{h}(\bar{m}_0)\|\geq b/2$ and $\|h_k(\bar{n}_0,n)-\tilde{h}(\bar{n}_0)\|\geq b/2$, for all  $n\in\M$ with $n>\bar{n}_0$.  In order to simplify notation, for each $n\in\N$, let \[u_n=h_k(\bar{m}_0,n)-\tilde{h}(\bar{m}_0)\ \ \text{and}\ \  v_n=h_k(\bar{n}_0,n)-\tilde{h}(\bar{n}_0).\] So, both $(u_n)_n$ and $(v_n)_n$ are weakly null. Hence, for  each $n\in\N$, we can pick $i(n)>n$ so that $\|u_n\|-\delta\leq \|u_n-v_{i(n)}\|$. So,  $\|u_n-v_{i(n)}\|\geq b/2-\delta$, for all $n\in\N$. 

Notice that $(v_{i(n)})_n$ is weakly null, hence, so is $(u_n-v_{i(n)})_n$.  Therefore, for a nonprincipal ultrafilter $\mathcal{U}$  on $\N$ and $h\in Y\setminus\{0\}$, it follows from the  definition of $\tilde{\delta}_Y$ that
\begin{align*}
\lim_{n,\mathcal{U}}\|h+(u_n-v_{i(n)})\|&\geq \|h\|\lim_{n,\mathcal{U}}\overline{\delta}_Y\Big(\frac{\|u_n-v_{i(n)}\|}{\|h\|}\Big)+\|h\|\\
&\geq \|h\|\lim_{n,\mathcal{U}} \tilde{\delta}_Y\Big(\frac{\|u_n-v_{i(n)}\|}{\|h\|}\Big)+\|h\|\\
&=\lim_{n,\mathcal{U}}N_2(\|h\|,\|u_n-v_{i(n)}\|).
\end{align*}
Hence, since $\|\sum_{j=1}^{k-1}(y_j-z_j)\|\geq  N_{k-1}(\frac{b}{2},\ldots,\frac{b}{2})-\delta$ and $\|u_n-v_{i(n)}\|\geq b/2-\delta$, for all $n\in\N$, by letting $h=\sum_{j=1}^{k-1}(y_j-z_j)$ we get $U\in\mathcal{U}$ so that  
\begin{equation}\label{Ineq}
\Big\|\sum_{j=1}^{k-1}(y_j-z_j)+(u_n-v_{i(n)})\Big\|\geq N_2\Big(N_{k-1}\Big(\frac{b}{2},\ldots,\frac{b}{2}\Big)-\delta,\frac{b}{2}-\delta\Big)-\delta,
\end{equation}
for all $n\in U$. As $\mathcal{U}$ is nonprincipal, $U$ is infinite. So,  $(u_n)_{n\in U}$ and $(v_{i(n)})_{n\in U}$
 are weakly null sequences. Therefore, we can pick $m_k\in U$, with $m_k>n_{k-1}$, and $y_k,z_k\in Y_k$ so that $\|u_{m_k}-y_k\|\leq \delta$ and $\|v_{i(m_k)}-z_k\|\leq \delta$. Hence, by Inequality \ref{Ineq} and by our choice of $\delta$, it follows that 
\[\Big\|\sum_{j=1}^{k-1}(y_j-z_j)+y_k-z_k\Big\|\geq N_{k}\Big(\frac{b}{2},\ldots,\frac{b}{2}\Big)-\eps.\]

At last, notice that
\begin{align*}
\Big\|h_k(\bar{m}_0,m_k&)-(y+y_1+\ldots+y_{k-1}+y_k)\Big\|\\
&\leq \Big\|\tilde{h}(\bar{m}_0)-(y+y_1+\ldots+y_{k-1})\Big\|+\Big\|h_k(\bar{m}_0, m_k)-\tilde{h}(\bar{m}_0)-y_k\Big\|\\
& \leq \delta+\delta\leq  \epsilon.
\end{align*}
Similarly, $\|h_k(\bar{n}_0, i(m_k))-(y+z_1+\ldots+z_k)\|\leq \eps$. Let $\bar{m}=\bar{m}_0\cup\{m_{k}\}$ and $\bar{n}=\bar{n}_0\cup\{i(m_{k})\}$. This finishes our induction. 
\end{proof}

\begin{proof}[Proof of Lemma \ref{lemmaGenNOVO}]
Fix $\eps>0$ and let   $y\in Y$ be given by applying Lemma \ref{lemmaGen} to $k$, $h_k$ and $\eps/3$. Hence, given an infinite subset  $\M\subset \N$, we can pick  $y_1,z_1,\ldots,y_k,z_k\in Y$ and $\bar{m}=(m_1,\ldots,m_k),\bar{n}=(n_1,\ldots,n_k)\in [\M]^k$, with $m_1<n_1<\ldots<m_k<n_k$, so that
\begin{align*}
\|h_k(\bar{m})-h_k(\bar{n})\|&\geq \|y_1-z_1+\ldots+ y_k- z_k\|\\
&\ \ \  \ -\|h(\bar{m})-(y+ y_1+\ldots+ y_k)\|\\
&\ \ \ \ -\|h(\bar{n})-(y+y_1+\ldots+ y_k)\|\\
&\geq N_k\Big(\frac{b}{2},\ldots,\frac{b}{2}\Big)-\eps.
\end{align*}
\end{proof}

Notice that the full strength of Lemma \ref{lemmaGen} has not been used in the proof of Lemma \ref{lemmaGenNOVO}. However, we make full use of Lemma \ref{lemmaGen} in Subsection \ref{SubsectionLowerEst} below.

We now have all the necessary tools needed to prove Theorem \ref{MainResult}.

\begin{proof}[Proof of Theorem \ref{MainResult}]
Let $f:X\to Y$ be a weakly sequentially continuous map which is coarse and satisfies Property $(*)$. Since $f$ satisfies Property $(*)$, we can pick $\beta>\alpha>0$ so that \[ \inf_{\|x-y\|\in[\alpha,\beta]}\|f(x)-f(y)\|>0.\] Let $(x_n)_n$ be a normalized weakly null sequence with spreading model $(e_n)_n$, so
\begin{align}\label{Eq1}
\|e_1+\ldots+e_k\|_S=\lim_{(n_1,\ldots,n_k)\to\infty}\|x_{n_1}+\ldots+x_{n_k}\|,
\end{align}
for all $k\in\N$. Without loss of generality, we can assume that $(x_n)_n$ is a $2$-basic sequence. Hence $\delta\coloneqq\inf_{n\neq m}\|x_n-x_m\|\geq 1/2$. So,  $(x_n)_n$ is $[\delta,2]$-separated (without loss of generality $\delta<2$). Let $\delta_0$ be a positive real number smaller than $\min\{2-\delta,\delta(\beta-\alpha)/\alpha\}$. Then, by standard Ramsey theory (see \cite{Todorcevic2010}, Theorem 1.3), we can assume that $(x_n)_n$ is $[a,a+\delta_0]$-separated, for some $a\in [\delta,2-\delta_0]$. Let $F(x)=f(\frac{\alpha}{a} x)$, for all $x\in X$. Then
\[\inf_{\|x-y\|\in[a,a+\delta_0]}\|F(x)-F(y)\|\geq b>0,\]
where  
\[b\coloneqq \inf_{\|x-y\|\in[\alpha,\beta]}\|f(x)-f(y)\|.\]
As $f$ is coarse, so is $F$. Let $L>0$ be such that 
\[\|F(x)-F(y)\|\leq L\|x-y\|\ \ \text{if}\ \ \|x-y\|\geq \frac{1}{2}.\] 
As $\delta\geq 1/2$, it is easy to see that $L$ does not depend on the sequence $(x_n)_n$.

Fix $k\in\N$ and pick a positive  $\eps<N_k(\frac{b}{4},\ldots,\frac{b}{4})$, where $N_k=N^{\tilde{\delta}_Y}_k$. Let $h_k:[\N]^k\to Y$ be given by $h_k(\bar{m})=F(x_{m_1}+\ldots+x_{m_k})$, for all $\bar{m}\in[\N]^k$.  By Lemma \ref{lemmaGenNOVO}, given any infinite subset  $\M\subset \N$, we can pick  $\bar{m}=(m_1,\ldots,m_k),\bar{n}=(n_1,\ldots,n_k)\in [\M]^k$, with $m_1<n_1<\ldots<m_k<n_k$, so that
\[\|h_k(\bar{m})-h_k(\bar{n})\|\geq N_k\Big(\frac{b}{2},\ldots,\frac{b}{2}\Big)-\eps\geq  N_k\Big(\frac{b}{4},\ldots,\frac{b}{4}\Big).\]
As $\M\subset\N$ is arbitrary, by standard Ramsey theory, there exists an infinite subset $\M\subset \N$ so that 
\[\|h_k(\bar{m})-h_k(\bar{n})\|\geq N_k\Big(\frac{b}{4},\ldots,\frac{b}{4}\Big),\]
for all $\bar{m}=(m_1,\ldots,m_k),\bar{n}=(n_1,\ldots,n_k)\in [\M]^k$ with $m_1<n_1<\ldots<m_k<n_k$.

On the other hand, as $(x_n)_n$ is $2$-basic, we have that $\|\sum_{j=1}^kx_{m_j}-\sum_{j=1}^kx_{m_j}\|\geq 1/2$, for all $\bar{m}\neq \bar{n}$. Therefore, we have that 
\[\|h_k(\bar{m})-h_k(\bar{n})\| \leq L\|x_{m_1}+\ldots+x_{m_k}-x_{n_1}-\ldots-x_{n_k}\|,\]
for all $\bar{m},\bar{n}\in[\M]^k$. Hence, for all $\bar{m}=(m_1,\ldots,m_k),\bar{n}=(n_1,\ldots,n_k)\in [\M]^k$ with $m_1<n_1<\ldots<m_k<n_k$, we have that
\[\frac{1}{L} N_k\Big(\frac{b}{4},\ldots,\frac{b}{4}\Big)\leq \|x_{m_1}-x_{n_1}+\ldots+x_{m_k}-x_{n_k}\| .\]
By Lemma \ref{Nkdes}, this gives us that 
\[ \frac{b}{4L}\|e_1+\ldots +e_k\|_{\tilde{\delta}_Y}\leq \|x_{m_1}-x_{n_1}+\ldots+x_{m_k}-x_{n_k}\|,\]
for all $\bar{m}=(m_1,\ldots,m_k),\bar{n}=(n_1,\ldots,n_k)\in [\M]^k$ with $m_1<n_1<\ldots<m_k<n_k$.

By Theorem \ref{Elton} and Equation \ref{Eq1}, we conclude that 
\[d\|e_1+\ldots +e_k\|_{\tilde{\delta}_Y}\leq \|e_1+\ldots+e_k\|_S,\]
for some $d>0$ depending only on $b$ and $L$, i.e., $d$ depends only on the map $f$, but it does not depend on either $k$ or on the sequence $(x_n)_n$. As $\tilde{\delta}_Y(t/2)\leq \overline{\delta}_Y(t)$, for all $t\geq 0$, we have that \[\frac{1}{2}\|e_1+\ldots +e_k\|_{\tilde{\delta}_Y}\leq \|e_1+\ldots +e_k\|_{\tilde{\delta}_Y},\] for all $k\in\N$.  This finishes the proof.
\end{proof}

\begin{proof}[Proof of Corollary \ref{CorMainResult}]
If $Y$ is $p$-AUC, then there exists $K>1$ so that $\overline{\delta}_Y(t)\geq Kt^p$, for all $t\in[0,1]$. Hence, as $\overline{\delta}_Y(t/2)\leq \tilde{\delta}_Y(t)$, we have that $\tilde{\delta}_Y(t)\geq (K/2^p)t^p$, for all $t\in[0,2]$. Using that $\tilde{\delta}_Y$ is increasing, we get  that \[\|e_1+\ldots+e_k\|_{\tilde{\delta}_Y}\geq \frac{K^{1/p}}{2}k^{1/p},\]  
for all $k\in\N$. So, by Theorem \ref{MainResult} (and Equation \ref{Ohoh}), it follows that $X$ has the $p$-co-Banach-Saks property.
\end{proof}

\begin{proof}[Proof of Corollary \ref{ellpdoesnotellq}]
It is easy to see that $\overline{\delta}_{\ell_r}(t)=(1+t^r)^{1/r}-1$, for all $r\in [1,\infty)$. Therefore, it easily follows  that $\ell_r$ is $r$-AUC, for all $r\in [1,\infty)$. However, $\ell_r$ does not have the  $s$-co-Banach-Saks property, for all $r\in (1,\infty)$ and all $s\in [1, r)$. So, the first claim follows from Corollary \ref{CorMainResult}. 

The last claim  follows from the fact that $\ell_q$ both coarsely and uniformly  embeds into $\ell_p$, for all $p,q\in [1,2]$ (see \cite{Nowak2006}, Theorem 5).
\end{proof}

\begin{remark}
We point out that, for $q>\max\{2,p\}$, $\ell_q$ does not coarsely (resp. uniformly) embed into $\ell_p$  (see \cite{MendelNaor2008}, Theorem 1.9 and Theorem 1.11).
\end{remark}

For $p\in [1,2]$, we know that $L_p$ strongly embeds into $\ell_p$ (see, for example, \cite{MendelNaor2004}, Remark 5.10, \cite{Nowak2006}, Theorem 5, and \cite{Randrianarivony2006}, page 1315). Since $L_p$ contains $\ell_2$, it follows that $L_p$ does not have the $p$-co-Banach-Saks property for any $p\in [1,2)$. Hence, Corollary \ref{CorMainResult} gives us that there is no  weakly sequentially continuous  coarse (resp. uniform) embedding of $L_p$ into $\ell_p$. Corollary \ref{CorL_p}(i) actually gives us a stronger result.

\begin{proof}[Proof of Corollary \ref{CorL_p}(i)]
By Corollary \ref{CorMainResult}, the existence of such embedding implies that $X$ has the $p$-co-Banach-Saks property. As a subspace of $L_p$ with the $p$-co-Banach-Saks property must be isomorphic to a subspace of $\ell_p$ (see \cite{Kalton2013}, Corollary 6.2), we are done.
\end{proof}

\begin{remark}
By Corollary \ref{CorMainResult}  above and Theorem 6.3 of \cite{Kalton2013}, we also have that if  $X$ is a subspace of a quotient of $L_p$ ($p\in (1,2)$) which admits either a coarse or a uniform embedding into $\ell_p$ by a  weakly sequentially continuous  map, then $X$ is a subspace of a quotient of $\ell_p$.
\end{remark}

In the case in which $X$ is an $\mathcal{L}_p$-space, we can get something stronger than Corollary \ref{CorL_p}(i).

\begin{proof}[Proof of Corollary \ref{CorL_p}(ii)]
This is a straightforward consequence of Corollary \ref{CorMainResult} and the fact that any $\mathcal{L}_p$-space is either isomorphic to $\ell_p$ or it contains an isomorphic copy of $\ell_2$ (see \cite{OdellJohnson1974}, Corollary 1).
\end{proof}

As mentioned in the introduction, to the best of our knowledge, it is not known whether an AUC Banach space must be $p$-AUC, for some $p\in[1,\infty)$. Therefore, Theorem \ref{c0spread2} cannot be seen as a consequence of Corollary \ref{CorMainResult}.

\begin{proof}[Proof of Theorem \ref{c0spread2}]
Let $(x_n)_n$ be a normalized weakly null sequence whose spreading model is isomorphic to the standard basis of $c_0$, call this spreading model $(e_n)_n$. Hence, there exists $M>0$ so that 
\[\|e_{1}+\ldots+e_{k}\|_S\leq M,\]
for all $k\in\N$. Then, by Theorem \ref{MainResult} (and Equation \ref{Ohoh}), there exists $c>0$ so that 
\[\|e_1+\ldots+e_k  \|_{\tilde{\delta}_Y}\leq \frac{1}{c} \|e_{1}+\ldots+e_{k}\|_S\leq \frac{M}{c},\]
for all $k\in\N$. As $\tilde{\delta}_Y$ is increasing, this gives us  that \[\sum_{i=1}^k\tilde{\delta}_Y\Big(\frac{c}{M}\Big)\leq\sum_{i=1}^k\tilde{\delta}_Y\Big(\frac{1}{\|e_1+\ldots+e_k  \|_{\tilde{\delta}_Y}}\Big)\leq 1,\]
for all $k\in\N$. So, $\tilde{\delta}_Y(c/M)=0$. Hence, $\tilde{\delta}_Y(c/(2M))=0$  and we conclude that $Y$ is not AUC. 

Let $(Y,\tn\cdot\tn)$ be a renorming of $(Y,\|\cdot\|)$ and let $\text{Id}:(Y,\|\cdot\|)\to (Y,\tn\cdot\tn)$ be the identity. As $\text{Id}\circ f$ is also a weakly sequentially continuous coarse map  satisfying Property $(*)$, this shows that $Y$ is not AUCable.
\end{proof}

\begin{problem}\label{problemAUC}
Let $X$ be an AUC space. Is $X$ $p$-AUCable, for some $p\in[1,\infty)$?
\end{problem}

Given a Banach space $X$,  the \emph{modulus of asymptotic uniform smoothness of $X$} is given by
\[\overline{\rho}_X(t)=\sup_{x\in\partial B_X}\inf_{E\in \text{cof}(X)}\sup_{y\in\partial B_E}\|x+ty\|-1.\]
$X$ is called \emph{asymptotically uniformly smooth} (\emph{AUS} for short) if $\lim_{t\to 0^+}\overline{\rho}_{X}(t)/t=0$. If $p\in(1,\infty)$, we say that $X$ is \emph{$p$-asymptotically uniformly smooth} (\emph{$p$-AUS} for short) if there exists $K>0$ so that $\overline{\rho}_X(t)\geq Kt^p$, for all $t\in [0,1]$. We define \emph{AUSable}  and \emph{$p$-AUSable} spaces analogously to AUCable and $p$-AUCable spaces. A Banach space $X$ is AUSable if and only if $X$ is $p$-AUSable for some $p>1$ (this was proved in \cite{KOS} for separable spaces, and later generalized for arbitrary Banach spaces in \cite{Raja2013}, Theorem 1,2). If $X$ is a separable reflexive Banach space, we have that $X$ is AUSable (resp. $p$-AUSable for some $p\in (1,\infty)$) if and only if $X^*$ is AUCable (resp. $p$-AUCable for some $p\in(1,\infty)$) (see \cite{OdellSchlumprecht2006}, Theorem 3). Therefore, Problem \ref{problemAUC} has a positive answer if $X$ is a separable reflexive space.

\begin{remark}
Notice that the map  $f$ in Theorem \ref{c0spread2} does not need to be defined in the whole $X$. Indeed, one can show that the same conclusion holds if $f$ is a bounded weakly sequentially continuous map defined on $M'\cdot B_X$, with $M'>M$, for which  there exists positive numbers $\alpha<\beta\leq 1$ so that $\inf_{\|x-y\|\in[\alpha,\beta]}\|f(x)-f(y)\|>0$.
\end{remark}

\subsection{Asymptotic structure}\label{SubsectionLowerEst}
It is clear from the proof of Theorem \ref{MainResult} that the full strength of Lemma \ref{lemmaGen} has not been used. In this subsection, we prove Theorem \ref{ThmAsymp} which fully uses Lemma \ref{lemmaGen}. In the case where $Y^*$ is separable,  this allows us to  obtain a result stronger than Corollary \ref{CorMainResult} above. For that, we need the machinery  of asymptotic games. 

Given a set $A$, let $A^{<\N}$ denote the set of finite tuples of elements of $A$. Let  $Y$ be a Banach space and $G\subset \partial B_Y^{<\N}$. For each $k\in\N$, we consider the following game $\mathcal{G}_k(Y,G)$ with $k$ rounds been played between two players, Player I and Player II:

\begin{align*}
&\text{Player 1 chooses }Y_1\in\text{cof}(Y).\\
&\text{Player 2 chooses }y_1\in \partial B_{Y_1}.\\
&\text{Player 1 chooses }Y_2\in\text{cof}(Y).\\
&\text{Player 2 chooses }y_2\in \partial B_{Y_2}.\\
&\vdots \\
&\text{Player 1 chooses }Y_k\in\text{cof}(Y).\\
&\text{Player 2 chooses }y_k\in \partial B_{Y_k}.\\
\end{align*}

For any $j\in\{1,\ldots,k\}$, we say that the tuple $(Y_1,y_1,\ldots,Y_j,y_j)$ is the \emph{outcome of the game after its $j$-th round}. We say that Player I wins the game if $(y_1,\ldots, y_k)\in G$. We say that Player I has a \emph{winning strategy  for $\mathcal{G}_k(Y,G)$} (WI-$\mathcal{G}_k(Y,G)$ for short) if Player I can force Player II to produce a tuple $(y_1,\ldots,y_k)$ in $G$. In other words, WI-$\mathcal{G}_k(Y,G)$ if
\[\exists Y_1\in\text{cof}(Y),\ \forall y_1\in \partial B_{Y_1},\ \ldots,\ \exists Y_n\in\text{cof}(Y),\ \forall y_k\in \partial B_{Y_k},\]
so that $(y_1,\ldots,y_k)\in G$.

\begin{defi}\label{DefinitionALE}
Let $p\in [1,\infty)$ and $Y$ be a Banach space. For $C>0$, we define   $G_{p,C}\subset \partial B_Y^{<\N}$  by setting
\[(y_1,\ldots,y_k)\in G_{p,C}\ \Leftrightarrow\ \Big\|\sum_{j=1}^ka_jy_j\Big\|\geq C\Big(\sum_{j=1}^k|a_j|^p\Big)^{1/p}, \ \ \text{for all} \ \ a_1,\ldots,a_k\in\R.\]
The Banach space $Y$ is said to satisfy \emph{asymptotic lower $\ell_p$-estimates} if there exists $C>0$ such that WI-$\mathcal{G}_k(Y,G_{p,C})$, for all $k\in\N$. 
\end{defi}

Let $p\in [1,\infty)$. In the class of spaces with separable dual, every   $p$-AUC space satisfies asymptotic lower $\ell_p$-estimates (see \cite{OdellSchlumprecht2002}, Proposition 5(a), and \cite{OdellSchlumprecht2006}, Proposition 2.3(e)). Therefore, for the class of spaces $Y$ with separable dual, the following result is a strengthening  of Corollary \ref{CorMainResult}.

\begin{thm}\label{ThmAsymp}
Let $X$ and $Y$ be  Banach spaces, and assume that $Y$  satisfies asymptotic lower $\ell_p$-estimates, for some $p\in [1,\infty)$.  Assume that there exists a   weakly sequentially continuous   map $f:X\to Y$   which is coarse and satisfies Property $(*)$. Then $X$ has the $p$-co-Banach-Saks property.
\end{thm}

\begin{proof}	
As $Y$ satisfies asymptotic lower $\ell_p$-estimates, we can pick $C>1$ so that $\text{WI-}\mathcal{G}_k(Y,G_{p,C})$, for all $k\in\N$. Let $(x_n)_n$ be a normalized weakly null sequence in $X$. Without loss of generality, we can assume that $(x_n)_n$ is a $2$-basic sequence. Hence $\delta\coloneqq\inf_{n\neq m}\|x_n-x_m\|\geq 1/2$. So,  $(x_n)_n$ is $[\delta,2]$-separated. Dilating  the argument in $f$ if necessary and proceeding just as in the proof of Theorem \ref{MainResult},  we can assume that  
\[b\coloneqq \inf_{\|x-y\|\in[\delta,2]}\|f(x)-f(y)\|>0.\]

Fix $k\in\N$, and let $h_k:[\N]^k\to Y$ be given by $h(\bar{m})=f(x_{m_1}+\ldots+x_{m_k})$, for all $\bar{m}\in[\N]^k$. Let   $y\in Y$ be given by applying Lemma \ref{lemmaGen} to $k$, $h_k$ and $\eps=bC/16$. Fix an infinite subset $\M\subset \N$. We now pick $Y_1,\ldots,Y_{k}\in\text{cof}(Y)$, $y_1,z_1,\ldots,y_{k},z_{k}\in Y$ and $\bar{m},\bar{n}\in[\M]^k$ as follows: let $Y_1\in \text{cof}(Y)$ be the space chosen by Player I's winning strategy in the first round of the game $\mathcal{G}_{k}(Y,G_{p,C})$ described above. Considering $\M\subset \N$ and given that $Y_1$ has already been  picked, we use Lemma \ref{lemmaGen}  to pick $y_1,z_1\in Y_1$. Let $j\in\{1,\ldots, k-1\}$, and assume that $Y_i\in\text{cof}(Y)$ and $y_i,z_i\in Y_i$ have been picked for all $i\leq j$. Let $Y_{j+1}\in\text{cof}(Y)$ be the space picked by Player I's winning strategy in the $(j+1)$-th round of the game $\mathcal{G}_{k}(Y,G_{p,C})$, considering that \[\Big(Y_1,\frac{y_1-z_1}{\|y_1-z_1\|},\ldots,Y_j,\frac{y_j-z_j}{\|y_j-z_j\|}\Big)\] is the outcome of the game $\mathcal{G}_{k}(Y,G_{p,C})$ after its $j$-th round is completed. We pick $y_{j+1},z_{j+1}\in Y_{j+1}$ using Lemma \ref{lemmaGen} applied to $\M$ and the tuple \[(Y_1,y_1,z_1,\ldots,Y_j,y_j,z_j,Y_{j+1}).\] Notice that, once $y_k$ and $z_k$ are picked, Lemma \ref{lemmaGen} also produces $\bar{m},\bar{n}\in[\M]^k$ with $m_1<n_1<\ldots<m_k<n_k$.  This completes the definition of $Y_1,\ldots,Y_{k}\in\text{cof}(Y)$,  $y_1,\ldots,y_{k}\in Y$, and $\bar{m},\bar{n}\in[\M]^k$. 

By our choice of $Y_1,\ldots,Y_{k}\in\text{cof}(Y)$, $y_1,z_1,\ldots,y_{k},z_{k}\in Y$ and $\bar{m},\bar{n}\in[\M]^k$, we have the following:
\begin{enumerate}[(i)]
\item $\|h(\bar{m})-(y+y_1+\ldots+y_k)\|\leq \eps$,
\item $\|h(\bar{n})-(y+z_{1}+\ldots+z_{k})\|\leq \eps$,
\item $\|y_i-z_i\|\geq  b/4$, for all $i\in\{1,\ldots, k\}$, and 
\item $\|y_1-z_1+\ldots+y_k-z_k\|\geq C(\sum_{i=1}^{k}\|y_i-z_i\|^p)^{1/p}$.
\end{enumerate}

Therefore, we have that
\begin{align*}
\|h(\bar{m})-h(\bar{n})\|&\geq \|y_1-z_1+\ldots+ y_k- z_k\|\\
&\ \ \  \ -\|h(\bar{m})-(y+ y_1+\ldots+ y_k)\|\\
&\ \ \ \ -\|h(\bar{n})-(y+z_1+\ldots+ z_k)\|\\
&\geq \frac{bC}{4}k^{1/p} -2\eps\\
&\geq \frac{bC}{8}k^{1/p}.
\end{align*}
As $\M\subset\N$ was arbitrary, this gives us that, for every infinite subset $\M\subset \N$, there exists $\bar{m},\bar{n}\in[\M]^k$, with $m_1<n_1<\ldots<m_k<n_k$, so that 
\[\|h(\bar{m})-h(\bar{n})\|\geq \frac{b C}{8}k^{1/p} .\]
Hence, by Ramsey theory, there exists an infinite $\M\subset \N$ so that 
\[\|h(\bar{m})-h(\bar{n})\|\geq \frac{bC}{8}k^{1/p} ,\]
for all $\bar{m},\bar{n}\in [\M]^k$ with $m_1<n_1<\ldots<m_k<n_k$. The proof now finishes just as the proof of Theorem \ref{MainResult}.
\end{proof}

\subsection{Weak$^*$ asymptotic uniform convexity.}\label{SubsectionWeakStar}

Let $X$ be a dual Banach space. We define the \emph{modulus of weak$^*$ asymptotic uniform convexity of $X$} by letting, for each $t\geq 0$,
\[\overline{\delta}^*_X(t)=\inf_{x\in\partial B_X}\sup_{E\in \text{cof}^*(X)}\inf_{y\in\partial B_E}\|x+ty\|-1,\]
where $\text{cof}^*(X)$ denotes the set of weak$^*$ closed finite codimensional subspaces of $X$.  We say that $X$ is \emph{weak$^*$ asymptotically uniformly convex} if $\overline{\delta}^*_X(t)>0$, for all $t>0$. For $p\in [1,\infty)$, we say that $X$ is \emph{$p$-weak$^*$-asymptotically uniformly convex} (\emph{$p$-weak$^*$-AUC} for short) if there exists $K>0$ so that $\overline{\delta}^*_X(t)\geq Kt^p$, for all $t\in[0,1]$. 

Similarly as with $\overline{\delta}_X$, it is easy to see that $\overline{\delta}^*_X$ is $1$-Lipschitz, and that for every weak$^*$ null sequence $(x_{n})_{n\in\N}$ in $X$, every nonprincipal ultrafilter $\mathcal{U}$ on $\N$ and every $x\in X\setminus \{0\}$, we have that
\[\|x\|\lim_{n,\mathcal{U}}\overline{\delta}^*_X\Big(\frac{\|x_\lambda\|}{\|x\|}\Big)\leq \lim_{n,\mathcal{U}}\|x+x_\lambda\|-\|x\|.\]
Also, as $\overline{\delta}^*_X(t)/t$ is increasing, we have that 
\[\tilde{\delta}^*_X(t)=\int_0^t\frac{\overline{\delta}^*_X(s)}{s}ds\]
is a convex function. Since, $\overline{\delta}^*_X(t/2)\leq \tilde{\delta}^*_X(t)\leq \overline{\delta}^*_X(t)$, for all $t\geq 0$, the functions $\overline{\delta}^*_X$ and $\tilde{\delta}^*_X$ are equivalent. 

Let $Y$ be a Banach space. Since $\tilde{\delta}_Y^*$ is a Lipschitz Orlicz function, we can construct a sequence of absolute norms $(N^*_k)_k$ by letting $N^*_k=N^{\tilde{\delta}^*_Y}_k$, for all $k\in\N$. If $X$ and $Y$ are dual spaces and $f:X\to Y$ is  weak$^*$ sequentially continuous,  Lemma \ref{lemmaGenNOVO} remains true with $(N^*_k)_k$ replacing $(N_k)_k$ in its conclusion. Also, under the same hypothesis, Lemma \ref{lemmaGen} remains true if we also replace   $\text{cof}(Y)$ by $\text{cof}^*(Y)$. Therefore, we obtain the following analogs of Theorem \ref{MainResult} and Theorem \ref{ThmAsymp}.

\begin{thm}\label{MainResultStarVersion}
Let $X$ and  $Y$ be a dual Banach spaces. Assume that $X$ maps into $Y$ by a weak$^*$ sequentially continuous map which is coarse and satisfies Property $(*)$.  Then there exists $c>0$, so that if $(e_n)_n$ is the spreading model in a Banach space $(S,\|\cdot\|_S)$ of a normalized weak$^*$ null sequence in $X$, then
\[c\|e_1+\ldots+e_k  \|_{\tilde{\delta}^*_Y}\leq\|e_1+\ldots+e_k\|_S,\]
for all $k\in\N$. 
\end{thm}

Given a Banach space $Y$, a set $G\subset \partial B_Y^{<\N}$ and $k\in\N$, we define the game $\mathcal{G}_k^*(Y,G)$ analogously to the game $\mathcal{G}_k(Y,G)$ but replacing $\text{cof}(Y)$ with $\text{cof}^*(Y)$ in its definition. Similarly, we say that Player I has a \emph{winning strategy for $\mathcal{G}_k^*(Y,G)$} (\text{WI-}$\mathcal{G}^*_k(Y,G)$ for short) if Player I can force Player II to produce a sequence $(y_1,\ldots,y_k)$ in $G$. We say that $Y$ satisfies \emph{weak$^*$ asymptotic lower $\ell_p$-estimates} if there exists $C>0$ so that \text{WI-}$\mathcal{G}^*_k(Y,G_{p,C})$, for all $k\in\N$, where $G_{p,C}$ is defined as in Definition \ref{DefinitionALE}. 

\begin{thm}\label{ThmAsympWeak}
Let $X$ and $Y$ be  dual Banach spaces, and assume that $Y$  satisfies weak$^*$ asymptotic lower $\ell_p$-estimates, for some $p\in [1,\infty)$.  Assume that there exists a   weak$^*$ sequentially continuous   map $f:X\to Y$   which is coarse and satisfies Property $(*)$. Then $X$ has the $p$-co-Banach-Saks property.
\end{thm}

It is clear that all the other results in this paper for weakly sequentially continuous coarse maps $X\to Y$ satisfying Property $(*)$ have analog versions for weak$^*$ sequentially continuous maps if the Banach spaces $X$ and $Y$ are dual spaces.

\section{Examples of  weakly sequentially  continuous  nonlinear embeddings.}\label{SectionExample}

In this section, we give a sufficient condition for a coarse map $X\to Y$ to be weakly sequentially continuous (under some assumptions on $Y$). We use this to show that there are Banach spaces $X$ and $Y$  so that $X$ strongly embeds into $Y$ by a map which is also weakly sequentially continuous, but $X$ does not embed into $Y$ isomorphically. 

Let $(X_n,\|\cdot\|_n)_n$ be a sequence of Banach spaces and let $\mathcal{E}=(e_n)_n$ be a $1$-unconditional basic sequence in a Banach space $(E,\|\cdot\|_E)$. We define the \emph{$\mathcal{E}$-sum of $(X_n,\|\cdot\|_n)_n$}, which we call $(\oplus _n X_n)_{\mathcal{E}}$, to be the space of sequences $(x_n)_n$, where $x_n\in X_n$, for all $n\in\N$, so that 
\[\|(x_n)_n\|\vcentcolon =\Big\|\sum_{n\in\N}\|x_n\|_ne_n\Big\|_E<\infty.\]
 One can check that $(\oplus _n X_n)_\mathcal{E}$ endowed with the norm $\|\cdot\|$ defined above is a Banach space. 

\begin{lemma}\label{lemmaImpliesWeakGEN}
Let $X$ be a Banach space, $(Y_n)_n$ be a sequence of Banach spaces and $\mathcal{E}$ be a $1$-unconditional shrinking basic sequence. For each $n\in\N$, let $f_n:X\to Y_n$ be a continuous  map satisfying the following property: there exists a finite rank operator $P_n:X\to X $ so that $f_n(x)=f_n(P_n(x))$, for all $x\in X$. Assume that the map $f=(f_n):X\to (\oplus_n Y_n)_{\mathcal{E}}$ given by $f(x)=(f_n(x))_n$, for all $x\in X$, is well defined and coarse. Then $f$ is weakly sequentially continuous. 
\end{lemma}

The following corollary is a straightforward application of Lemma \ref{lemmaImpliesWeakGEN}.

\begin{cor}\label{lemmaImpliesWeak}
Let $(X_n)_n$ and $(Y_n)_n$ be  sequences of Banach spaces. Let $\mathcal{F}$  and $\mathcal{E}$ be  $1$-unconditional basic sequences and assume that $\mathcal{E}$ is shrinking. For each $n\in\N$, let $f_n:X_n\to Y_n$ be a continuous map and  assume that the map $f=(f_n):(\oplus_n X_n)_{\mathcal{F}}\to (\oplus_n Y_n)_{\mathcal{E}}$ given by $f((x_n)_n)=(f_n(x_n))_n$, for all $(x_n)_n\in (\oplus_n X_n)_{\mathcal{E}}$, is well defined and coarse. Then $f$ is weakly sequentially continuous. 
\end{cor}

\begin{proof}[Proof of Lemma \ref{lemmaImpliesWeakGEN}]
Say $\mathcal{E}=(e_n)_n$ and let $f=(f_n):X\to (\oplus_n Y_n)_{\mathcal{E}}$ be as above. Let $(x_m)_{m}$ be a sequence in $X$ weakly converging to an element $x\in X$. Let us show that $(f(x_m))_{m}$ weakly converges to $f(x)$. As $(x_m)_{m}$ is weakly convergent, it must be bounded, hence, as $f$ is coarse, it follows that $(f(x_m))_{m}$ is bounded. Let \[M=\max\{\|f(x)\|,\sup_{m}\|f(x_m)\|\}.\]

For each $n\in\N$, let  $P_n:X\to X$ be the finite rank operator given by the hypothesis in the lemma. As $P_n$ is linear and norm continuous, it follows that $(P_n(x_m))_m$ weak converges to $P_n(x)$. As, $P_n(X)$ is finite dimensional, $(P_n(x_m))_m$ converges to $P_n(x)$ in norm. As $f_n$ is continuous, we have that $\lim_m f_n(P_n(x_m))=f_n(P_n(x))$.

Fix a functional $\varphi$ in the dual of $(\oplus_n Y_n)_{\mathcal{E}}$ and $\eps>0$. We can write $\varphi=(\varphi_n)_n$, with $\varphi_n\in Y^*_n$, for all $n\in\N$. For each $N\in\N$, let $\mathcal{E}_N=(e_{n+N})_n$ and let
\[\|(\varphi_n)_{n>N}\|\coloneqq \sup\Big\{\Big|\sum_{n=1}^\infty\varphi_{n+N}(x_n)\Big|\mid (x_n)_n\in B_{(\oplus_n Y_{n+N})_{\mathcal{E}_N}}\Big\}.\]
As $\mathcal{E}$ is shrinking, it follows that $\lim_{N\to \infty}\|(\varphi_n)_{n>N}\|=0$. Pick $N\in\N$ so that $\|(\varphi_n)_{n>N}\|<\eps/(4M)$. As $\lim_m f_n(P_n(x_m))=f_n(P_n(x))$, for all $n\in\N$, we can pick $m_0\in \N$, so that $\|f_n(P_n(x_m))-f_n(P_n(x))\|<\eps/(2NL)$, for all $m>m_0$ and all $n\in\{1,\ldots,N\}$, where $L=\max_{n<N}\|\varphi_n\|$. Hence, we have that

\begin{align*}
|\varphi(f( x_{ m} ))-\varphi(f(x))|&=\Big|\sum_{n=1}^\infty\varphi_n(f_n(x_m))-\sum_{n=1}^\infty\varphi_n(f_n(x))\Big|\\
&\leq \sum_{n=1}^N|\varphi_n(f_n(P_n(x_m)))-\varphi_n(f_n(P_n(x)))|\\
& \ \ \ \ +\sum_{n=M}^\infty|\varphi_n(f_n(x_m))-\varphi_n(f_n(x))|\\
&\leq \sum_{n=1}^N\|\varphi_n\|\cdot \|f_n(P_n(x_m))-f_n(P_n(x))\|\\
&\ \ \ \ +\Big\|(\varphi_n)_{n>N}\Big\|\cdot\Big\|f_n(x_m)-f_n(x)\Big\|\\
&\leq \eps,
\end{align*}
for all $m>m_0$.  Hence, $f$ is  weakly sequentially continuous and we are done. 
\end{proof}

The following was proved in \cite{AlbiacBaudier2015}, Theorem 3.4.

\begin{thm}[\textbf{F. Albiac and F. Baudier, 2015}]\label{AlbiacBaudier} Let $p,q\in(0,\infty)$, and assume that $p<q$. There exist constants $A,B>0$ and a sequence of real-valued maps $(\psi_j)_{j=1}^\infty$ so that 
\[A|x-y|^p\leq \sum_{j\in\N}|\psi_j(x)-\psi_j(y)|^q\leq B|x-y|^p,\]
for all $x,y\in \R$.
\end{thm}

\begin{proof}[Proof of Theorem \ref{222}]
Let $A$, $B$ and $(\psi_j)_{j=1}^\infty$ be given by Theorem \ref{AlbiacBaudier} applied to $p,q\in [1,\infty)$. As noticed in Remark 3.5 of \cite{AlbiacBaudier2015}, the map $f:\ell_p\to \ell_q(\ell_q)$ given by 
\[f((x_n)_{n=1}^\infty)=\Big(\big(\psi_j(x_n)-\psi_j(0)\big)_{j=1}^\infty\Big)_{n=1}^\infty, \]
for all $(x_n)_{n=1}^\infty\in \ell_p$, is a strong embedding. Indeed, it is straightforward to check that
\[A^{1/q}  \|x-y\|^{p/q}_{\ell_p} \leq  \|f(x)-f(y)\|_{\ell_q(\ell_q)}\leq B^{1/q} \|x-y\|^{p/q}_{\ell_p}\]
for all $x,y\in\ell_p$. Therefore, as $\ell_q(\ell_q)\equiv \ell_q$,  we only need to show that $f$ is  weakly sequentially continuous. 

Let $X_n=\R$ and $Y_n=\ell_q$, for all $n\in\N$, and let $\mathcal{F}$ and $\mathcal{E}$ be the standard basis of $\ell_p$ and $\ell_q$, respectively. Applying  Corollary \ref{lemmaImpliesWeak}, we are done.
\end{proof}

\begin{remark}
Notice that, by the same arguments above, the map $f:\ell_p\to \ell_q(\ell_q)$ defined in the proof of Theorem \ref{222} is not only weakly sequentially continuous, but it  is also weakly continuous when restricted to any bounded subset of $ \ell_p$. However, $f:\ell_p\to \ell_q(\ell_q)$ is not weakly continuous. Indeed, by the proof of Theorem 3.4 of \cite{AlbiacBaudier2015}, we have that 
\[\psi_j(t)=\left\{\begin{array}{ll}
2^{j_1(1-\frac{p}{q})-1}\Big(t-\frac{j_2-2}{2^{j_1}}\Big), & \ \ \text{if}\ \ \frac{j_2-2}{2^{j_1}}\leq t< \frac{j_2}{2^{j_1}},\\
-2^{j_1(1-\frac{p}{q})-1}\Big(t-\frac{j_2-2}{2^{j_1}}\Big), & \ \ \text{if}\ \  \frac{j_2}{2^{j_1}}\leq t< \frac{j_2+2}{2^{j_1}},
\end{array}\right. \] 
where $j\in\N\mapsto (j_1,j_2)\in\Z\times\Z$ is a fixed bijection. Hence, for every $n\in\N$, we can  choose $j(n)\in\N$ so that $\psi_{j(n)}(0)=0$ and $\sup_t\psi_{j(n)}(t)>n$. Define $\varphi=(a_{n,j})_{n,j}$ by letting, for all $n,j\in\N$, 
\[a_{n,j}=\left\{\begin{array}{ll}
\frac{1}{n}, &\ \ \text{if}\ \  j=j(n),\\
0, &\ \ \text{otherwise.} 
\end{array}\right. \]
Hence, $\varphi=((a_{n,j})_j)_n\in \ell_{q'}(\ell_{q'})\equiv (\ell_q(\ell_q))^*$, where $q'$ is the conjugate of $q$, i.e., $1/q+1/q'=1$. Make $\text{Cof}\coloneqq \text{cof}(\ell_p)$ into a directed set by setting $V\preceq U$ if $U\subset V$, for all $V,U\in \text{Cof}$. For each $V\in \text{Cof}$ pick $x_V\in V\setminus \{0\}$. As $x_V\neq 0$, there exists $n(V)\in\N$ so that $x_V(n(V))\neq 0$. Hence, dilating $x_V$, we can assume that $\psi_{j(n(V))}(x_V(n(V)))>n(V)$, for all $V\in \text{Cof}$. Clearly, $(x_V)_{V\in\text{Cof}}$ is weakly null. However, we have that
\begin{align*}
\varphi(f(x_V))&=\sum_{n,j}a_{n,j}(\psi_j(x_V(n))-\psi_j(0))\\
&=\sum_na_{n,j(n)}(\psi_{j(n)}(x_V(n))-\psi_{j(n)}(0))\\
&=\sum_n\frac{1}{n}\psi_{j(n)}(x_V(n))>1,
\end{align*}
for all $V\in \text{Cof}$. Therefore, $(\varphi(f(x_V))_{V\in \text{Cof}}$ is not weakly null. 
\end{remark}

We finish this section noticing that, by Corollary 4.14 of \cite{Braga2}, we have that the Tsirelson space $T$ constructed by T. Figiel and W. Johnson (see \cite{FigielJohnson1974}) strongly embeds into a superreflexive Banach space. Precisely, $T$ strongly embeds into $(\oplus T^2)_{T^2}$, i.e., the unconditional sum of the $2$-convexification of the Tsirelson space $T$ with respect to its standard unit basis (see \cite{Braga2}, Section 2, for definitions). Proceeding as in Corollary 4.14 of \cite{Braga2} and applying Lemma \ref{lemmaImpliesWeakGEN}, one obtains that the strong embedding $T\to (\oplus T^2)_{T^2}$ above is also weakly sequentially continuous. However, as superreflexive Banach spaces are AUCable and as $T^*$ has $c_0$ as a spreading model, Theorem \ref{c0spread2} gives us that if $T^*$ is mapped into a Banach space $X$ by a weakly sequentially continuous coarse map satisfying Property $(*)$, then $X$ is not superreflexive. This result relates to the main theorem of \cite{BaudierLancienSchlumprecht2017} (see \cite{BaudierLancienSchlumprecht2017},  Theorem C), which gives us in particular that if a Banach space $X$ coarsely embeds into $T^*$, then $X$ is not superreflexive.
 
\section{Problems.}

We end this paper with some questions that naturally arise from our results. Firstly, we notice that, although we have shown that weakly  sequentially continuous coarse maps satisfying Property $(*)$ is a new kind of embedding between Banach spaces, we do not know if the same holds assuming weak continuity.

\begin{problem}
Can we find Banach spaces $X$ and $Y$ so that $X$ does not embed into $Y$ isomorphically, but $X$ can be mapped into $Y$ by a weakly  continuous coarse map satisfying Property $(*)$? Is there a weakly continuous map $\ell_p\to \ell_q$ which is coarse and satisfies Property $(*)$, for $1\leq p<q$? 
\end{problem}

It would be interesting to obtain some strengthening  of  Corollary \ref{CorL_p}.

\begin{problem}
Let $X$ be a Banach space with type $p\in [1,2]$. Assume that $X$  coarsely (resp. uniformly) embeds into $\ell_p$ by a  weakly sequentially continuous  map. Does it follow that $X$ is isomorphic to a subspace of $\ell_p$?
\end{problem}

Although Lemma \ref{lemmaImpliesWeakGEN} gives us a sufficient   condition for when a coarse map $f:X\to Y$ is weakly sequentially continuous, the hypothesis  in Lemma \ref{lemmaImpliesWeakGEN} are very restrictive and it requires the Banach space $Y$ to be of a very specific  type. It would be interesting to obtain more general results along this line.

\begin{problem}
Let $X$ and $Y$ be Banach spaces. Under which conditions does coarse (resp. uniform) embeddability  of $X$ into $Y$ implies the existence of a  weakly sequentially continuous coarse (resp. uniform) embedding $X\to Y$?
\end{problem}

Notice that, if $X$ is a Schur space, then any continuous map $X\to Y$ is weakly sequentially continuous. Therefore $X$ coarsely embeds into a Banach space $Y$ if and only if it coarsely embeds into $Y$ by a map which is weakly sequentially continuous (see \cite{Braga1}, Theorem 1.4). However, since a Schur space always contains $\ell_1$, Theorem \ref{MainResult} is not capable of giving us restrictions on coarse embeddability of $X$ into any Banach space.

In these notes, we only studied weakly sequentially continuous maps $X\to Y$ which are coarse and satisfy Property $(*)$. We could also introduce the notion of weakly sequentially coarse equivalent Banach spaces. Precisely, we say that two Banach spaces $X$ and $Y$ are \emph{weakly sequentially coarse equivalent} if there exists a bijection $f:X\to Y$ so that both $f$ and $f^{-1}$ are coarse and weakly sequentially continuous. 

\begin{problem}
Let $X$ and $Y$ be weakly sequentially coarsely equivalent Banach spaces.  Are $X$ and $Y$ linearly isomorphic?
\end{problem}

\textbf{Acknowledges:} I would like to thank Christian Rosendal for fruitful conversations and suggestions which helped to improve this paper.

\end{document}